\documentclass[10pt,runningheads,a4paper]{llncs}
\usepackage{graphicx}
\usepackage{amsmath}
\usepackage{amssymb}
\usepackage{color}
\numberwithin{theorem}{subsection}
\numberwithin{corollary}{section}
\numberwithin{definition}{subsection}
\begin{document}
\title{On Topological Properties of Third type of Hex Derived Networks}
\author{Haidar Ali$^1$, Muhammad Ahsan Binyamin$^1$, Muhammad Kashif Shafiq$^1$}

\institute{$^1$Department of Mathematics, \\Government College University,\\
Faisalabad, Pakistan\\
E-mail: haidarali@gcuf.edu.pk, \{ahsanbanyamin, kashif4v\}@gmail.com}
\authorrunning{H. Ali, M. A. Binyamin, M. K. Shafiq}
\titlerunning{On Topological Properties of Third type of Hex Derived Networks} \maketitle

\markboth{\small{H. Ali, , M. A. Binyamin, M. K. Shafiq}} {\small{On Topological Properties of Third type of Hex Derived Networks}}

\begin{abstract}
In chemical graph theory, a topological index is a numerical representation of a chemical network while a topological descriptor correlates certain physico-chemical characteristics of underlying chemical compounds besides its chemical representation. Graph plays an vital role in modeling and designing any chemical network.\\
F. Simonraj et al. derived new third type of Hex derived networks \cite{Simonraj new}. In our work, we discuss the third type of hex derived networks $HDN3(r)$, $THDN3(r)$ and $RHDN3(r)$ and computed exact results for topological indices which are based on degrees of end vertices.
\end{abstract}
\par

\vspace{.3cm} {\bf Keywords:} General Randi\'{c} index, Harmonic index, Augmented Zagreb index, Atom-bond connectivity $(ABC)$ index, Geometric-arithmetic
$(GA)$ index, Third type of Hex Derived Networks, $HDN3(r)$, $THDN3(n)$, $RHDN3(n)$.

\section[Introduction and preliminary results]{Introduction and preliminary results}
Topology indices application is now a standard procedure in studying chemical information, structure properties like QSAR and QSPR. Biological indicators such as the Randi$\acute{c}$ Index, Zagreb Index, the Weiner Index and the Balaban index are used and predict and study the physical and chemical properties. There is too much research has been published in this field in the last few decades. The topological index is a numeric quantity associated with chemical constitutions purporting the correlation of chemical structures with many physicochemical properties, chemical reactivity or biological activity. Topological indices are made on the grounds of the transformation of a chemical network into a number that characterizes the topology of the chemical network.  Some of the major types of topological indices of graphs are distance-based topological indices, degree-based topological indices, and counting-related topological indices.
For any Graph $\mathrm{G}=(V, E)$ where V is be the vertex set and E to be the edge set of $\mathrm{G}$. The degree $\kappa(x)$ of vertex x is the amount of edges of $\mathrm{G}$ episode with x. A graph can be spoken by a polynomial, a numerical esteem or by network shape.

Hexagonal mesh was derived by Chen et al. \cite{Chen}.  A set of triangles, made up a \emph{hexagonal mesh} as shown in Fig. 1. Hexagonal mesh with dimension 1 does not exist. Composition of six triangles made a 2-dimensional hexagonal mesh $HX(2)$ (see Fig. 1(1)). By adding a new layer of triangles around the boundary of $HX(2)$, we have a 3-dimensional hexagonal mesh $HX(3)$ (see Fig. 1(2)). Similarly, we form $HX(n)$ by adding n layers around the boundary of each proceeding hexagonal mesh.\\
\noindent\textbf{Drawing algorithm of $HDN3$ networks}\\
Step-1: First we draw a Hexagonal network of dimension $r$.\\
Step-2: Replace all $K_{3}$ subgraph in to Planar octahedron $POH$ once. The resulting graph is called $HDN3$ (see Fig. \ref{fig 2}) networks.\\
Step-3: From $HDN3$ network, we can easily form $THDN3$ (see Fig. \ref{fig 3})and $RHDN3$ (see Fig. \ref{fig 4}).
\begin{figure}
\centering
  \includegraphics[width=8cm]{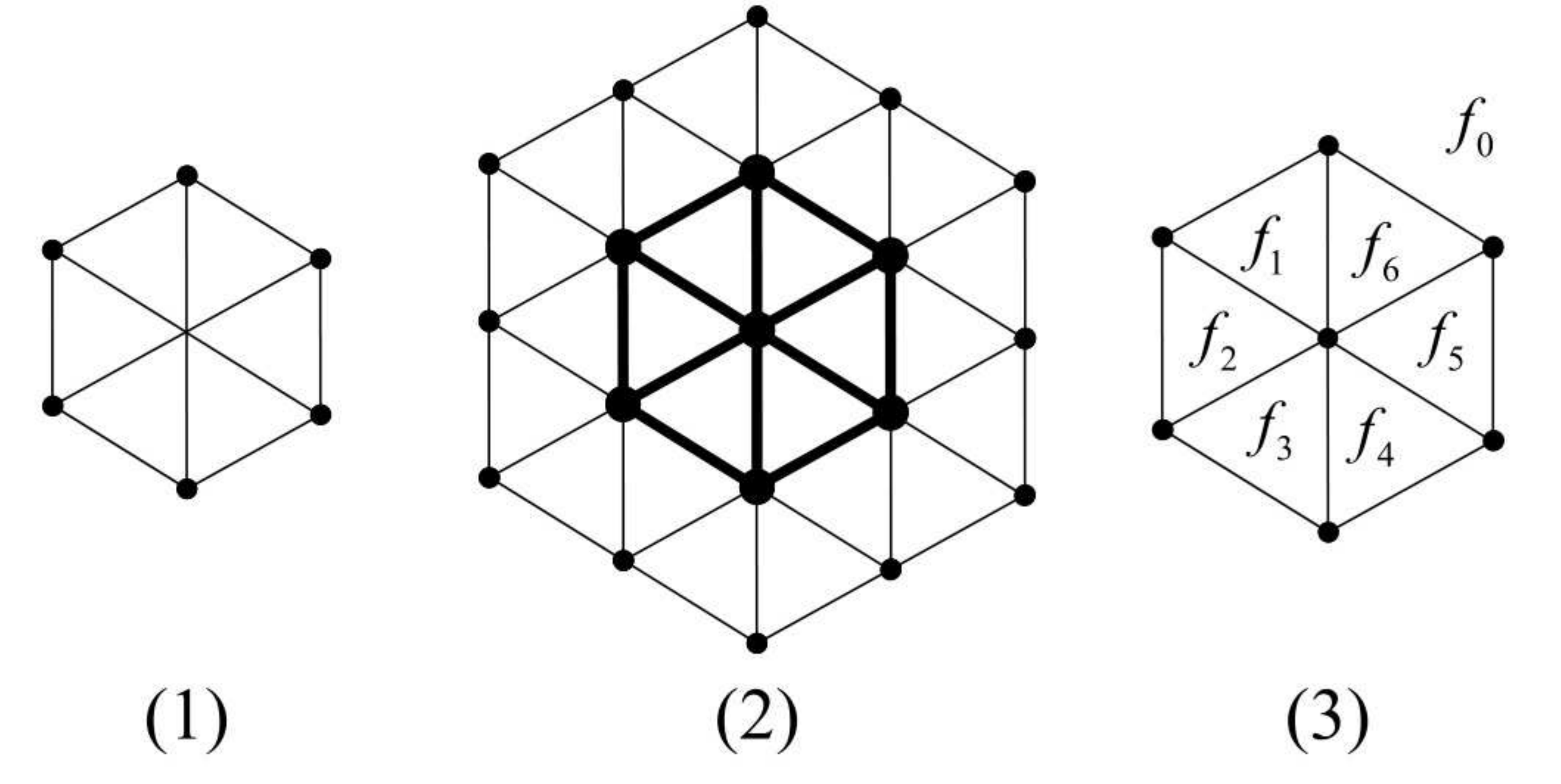}\\
  \caption{Hexagonal meshes: (1) HX(2) , (2) HX(3), and (3) all faces in HX(2).}\label{fig 1}
\end{figure}

In this paper, we consider $\mathrm{G}$ as a network with $V(\mathrm{G})$ is the set of vertices and edge set $E(\mathrm{G})$, the degree of any vertex $\acute{p}\in V(\mathrm{G})$ is denoted by $\kappa(\acute{p})$. The $S_{\acute{p}}={\sum\limits_{\acute{p}\in N_{\mathrm{G}}(\acute{p})}deg(\acute{q})}$ where $N_{\mathrm{G}}(\acute{p})=\{\acute{q}\in V(\mathrm{G})\mid \acute{p}\acute{q} \in E(\mathrm{G})\}$.

The beginning of degree based topological indices starts from \emph{Randi\'{c}} index \cite{Randic} denoted by $R_{-\frac{1}{2}}(\mathrm{G})$ and acquainted by Milan Randi\'{c} and written as
\begin{equation}
 R_{-\frac{1}{2}}(\mathrm{G})=\sum\limits_{\acute{p}\acute{q} \in \mathrm{E}(\mathrm{G})}\frac{1}{\sqrt{\kappa(\acute{p})\kappa(\acute{q})}}.
\end{equation}
B. Furtula and Ivan Gutman \cite{Furtula2} introduced forgotten topological index (also called F-Index) defined as
\begin{equation}
F(G)=\sum_{\acute{p}\acute{q} \in \mathrm{E}(\mathrm{G})}((\kappa(\acute{p}))^2+(\kappa(\acute{q}))^2).
\end{equation}
Another topological index in view of the level of the vertex is the Balaban index \cite{Balaban1,Balaban2}. This index for a graph $\mathrm{G}$ of order $`n'$, size $`m'$ is characterized as
\begin{equation}
 J(\mathrm{G})=\bigg(\frac{m}{m-n+2}\bigg)\sum_{\acute{p}\acute{q} \in \mathrm{E}(\mathrm{G})}\frac{1}{\sqrt{\kappa(\acute{p})\times \kappa(\acute{q})}}.
\end{equation}
Ranjini et al. \cite{Ranjini} reclassified the Zagreb indices is to be specific the re-imagined in the first place, second and third Zagreb indices for graph $\mathrm{G}$ as
\begin{equation}
ReZG_{1}(\mathrm{G})=\sum_{\acute{p}\acute{q} \in \mathrm{E}(\mathrm{G})}\bigg(\frac{\kappa(\acute{p})\times \kappa(\acute{q})}{\kappa(\acute{p})+ \kappa(\acute{p})}\bigg),
\end{equation}
\begin{equation}
ReZG_{2}(\mathrm{G})=\sum_{\acute{p}\acute{q} \in \mathrm{E}(\mathrm{G})}\bigg(\frac{\kappa(\acute{p})+\kappa(\acute{q})}{\kappa(\acute{p})\times \kappa(\acute{p})}\bigg),
\end{equation}
\begin{equation}
ReZG_{3}(\mathrm{G})=\sum_{\acute{p}\acute{q} \in \mathrm{E}(\mathrm{G})}(\kappa(\acute{p})\times\kappa(\acute{q}))(\kappa(\acute{p})+ \kappa(\acute{p}).
\end{equation}
Only $ABC_{4}$ and $GA_{5}$ indices can be computed if we are able to find the edge partition of these networks, based on sum of the degrees of end vertices of each edge. The fourth version of $ABC$ index is introduced by Ghorbani \emph{et al.} \cite{Graovac1} and defined as
\begin{equation}
ABC_{4}(\mathrm{G})=\sum_{\acute{p}\acute{q} \in \mathrm{E}(\mathrm{G})}\sqrt{\frac{S_{\acute{p}}+S_{\acute{q}}-2}{S_{\acute{p}}S_{\acute{q}}}}.
\end{equation}
The fifth version of $GA$ index is proposed by Graovac \emph{et al.}\cite{Graovac2} and defined as
\begin{equation}
GA_{5}(\mathrm{G})=\sum_{\acute{p}\acute{q} \in \mathrm{E}(\mathrm{G})}\frac{2\sqrt{S_{\acute{p}}S_{\acute{q}}}}{(S_{\acute{p}}+S_{\acute{q}})}.
\end{equation}

\section[Main Results for third type of Hex Derived networks]{Main Results for third type of Hex Derived networks}
F. Simonraj et al. \cite{Simonraj new} derived new third type of Hex derived networks and found metric dimension of $HDN3$ and $PHDN3$. In this work, we discuss the newly derived third type of hex derived networks and compute exact results for degree based topological indices. In these days, there is an extensive research activity on these topological indices and their variants see, \cite{Baca1,baig1,baig,Caporossi,Imran,Imran1,Imran2,chachaIranmanesh,Liu,Liu1,Simonraj}. For basic definitions, notations, see \cite{Bondy,Diudea,Gutman,Gutman1,Wiener}

\subsection{Results for Hex Derived network of type 3, $HDN3(r)$}
In this section, we discuss the newly derived third type of hex derived network and compute numerical and exact results for forgotten index, Balaban index, reclassified the Zagreb indices, forth version of $ABC$ index and fifth version of $GA$ index for the very first time.

\begin{theorem}
Consider the Hex Derived network of type 3 $HDN3(n)$, then its forgotten index is equal to
\begin{equation*}
F(HDN3(n))=6(5339-8132n+3108n^2)
\end{equation*}
\end{theorem}
\begin{figure}
\centering
  \includegraphics[width=5cm]{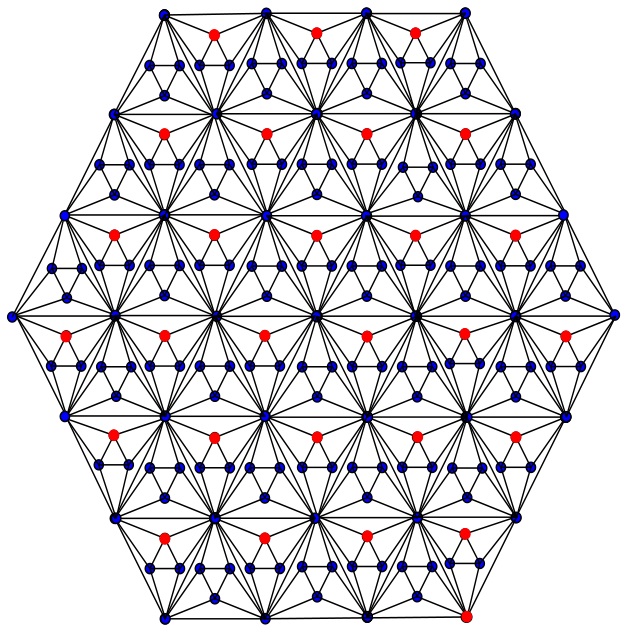}\\
  \caption{Hex Derived network of type 3 ($HDN3(4)$)}\label{fig 2}
\end{figure}
\vspace{0cm}
\begin{table}
\centering
\begin{tabular}{|c|c|c|c|c|}
   \hline
  $(\kappa_{x},\kappa_{y})$ where $\acute{p}\acute{q}\in E(\mathrm{G_1})$ & $\text{Number of edges}$&$(\kappa_{u},\kappa_{v})$ where $\acute{p}\acute{q}\in E(\mathrm{G_1})$&$\text{Number of edges}$ \\\hline
   $(4,4)$ & $18r^2-36r+18$&$(7,18)$ & $6$ \\\hline
   $(4,7)$ & $24$&$(10,10)$ & $6r-18$ \\\hline
   $(4,10)$ & $36r-72$&$(10,18)$ & $12r-24$\\\hline
   $(4,18)$ & $36r^2-108r+84$& $(18,18)$ & $9r^2-33r+30$\\\hline
   $(7,10)$ & $12$&&\\\hline
\end{tabular}
\vspace{.2cm}
\caption{Edge partition of Hex Derived network of type 3 $HDN3(r)$ based on degrees of end vertices of each edge.}
\end{table}
\vspace{-1cm}
\begin{proof}
Let $\mathrm{G_1}$ be the Hex Derived network of type 3, $HDN3(r)$ shown in Fig. \ref{fig 2}, where $r\geq4$. The Hex Derived network $\mathrm{G_1}$ has $21r^2-39r+19$ vertices and the edge set of $\mathrm{G_1}$ is divided into nine partitions based on the degree of end vertices. The first edge partition $E_1(\mathrm{G_1})$ contains $18r^2-36r+18$ edges $\acute{p}\acute{q}$, where $\kappa(\acute{p})=\kappa(\acute{q})=4$. The second edge partition $E_2(\mathrm{G_1})$ contains $24$ edges $\acute{p}\acute{q}$, where $\kappa(\acute{p})=4$ and $\kappa(\acute{q})=7$. The third edge partition $E_3(\mathrm{G_1})$ contains $36r-72$ edges $\acute{p}\acute{q}$, where $\kappa(\acute{p})=4$ and $\kappa(\acute{q})=10$. The fourth edge partition $E_4(\mathrm{G_1})$ contains $36r^2-108r+84$ edges $\acute{p}\acute{q}$, where $\kappa(\acute{p})=4$ and $\kappa(\acute{q})=18$. The fifth edge partition $E_5(\mathrm{G_1})$ contains $12$ edges $\acute{p}\acute{q}$, where $\kappa(\acute{p})=7$ and $\kappa(\acute{q})=10$. The sixth edge partition $E_6(\mathrm{G_1})$ contains $6$ edges $\acute{p}\acute{q}$, where $\kappa(\acute{p})=7$ and $\kappa(\acute{q})=18$. The seventh edge partition $E_7(\mathrm{G_1})$ contains $6r-18$ edges $\acute{p}\acute{q}$, where $\kappa(\acute{p})=\kappa(\acute{q})=10$ and the eighth edge partition $E_8(\mathrm{G_1})$ contains $12r-24$ edges $\acute{p}\acute{q}$, where $\kappa(\acute{p})=10$ and $\kappa(\acute{q})=18$ and the ninth edge partition $E_9(\mathrm{G_1})$ contains $9r^2-33r+30$ edges $\acute{p}\acute{q}$, where $\kappa(\acute{p})=\kappa(\acute{q})=18$, Table 1 shows such an edge partition of $\mathrm{G_1}$. Thus from $(3)$ this follows that
\begin{equation*}
F(G)=\sum_{\acute{p}\acute{q} \in \mathrm{E}(\mathrm{G})}((\kappa(\acute{p}))^2+(\kappa(\acute{q}))^2)
\end{equation*}
Let $\mathrm{G_1}$ be the Hex Derived network of type 3, $HDN3(r)$. By using edge partition from Table 1, the result follows. The forgotten index can be calculated by using (2) as follows.
\begin{equation*}
  F(\mathrm{G_1})=\sum_{\acute{p}\acute{q} \in \mathrm{E}(\mathrm{G})}((\kappa(\acute{p}))^2+(\kappa(\acute{q}))^2)=\sum_{\acute{p}\acute{q} \in \mathrm{E_{j}}(\mathrm{G})}\sum_{j=1}^9((\kappa(\acute{p}))^2+(\kappa(\acute{q}))^2)
\end{equation*}
\begin{eqnarray*}
F(\mathrm{G_1})&=&32|E_1(\mathrm{G_1})|+65|E_2(\mathrm{G_1})|+116|E_3(\mathrm{G_1})|+340|E_4(\mathrm{G_1})|+149|E_5(\mathrm{G_1})|+
\\&&373|E_6(\mathrm{G_1})|+200|E_7(\mathrm{G_1})|+424|E_8(\mathrm{G_1})|+648|E_9(\mathrm{G_1})|
\end{eqnarray*}
\end{proof}
By doing some calculations, we get\\
$$\Longrightarrow F(\mathrm{G_1})=6(5339-8132n+3108n^2).$$
\qed

In the following theorem, we compute Balaban index of Hex Derived network of type 3, $\mathrm{G_1}$.
\begin{theorem}
For Hex Derived network $\mathrm{G_1}$, the Balaban index is equal to
\begin{eqnarray*}
J(\mathrm{G_1})&=&\frac{(20-41n+21n^2)(1595.47+7(-307-270\sqrt{2}+12\sqrt{5}+54\sqrt{10})n)}{70(43-84n+42n^2)}+\\&&
\frac{210(5+3\sqrt{2})n^2}{70(43-84n+42n^2)}
\end{eqnarray*}
\end{theorem}
\begin{proof}
Let $\mathrm{G_1}$ be the Hex Derived network $HDN3(r)$. By using edge partition from Table 1, the result follows. The Balaban index can be calculated by using (3) as follows.
\begin{eqnarray*}
 J(\mathrm{G_1})&=&\bigg(\frac{m}{m-n+2}\bigg)\sum_{\acute{p}\acute{q} \in \mathrm{E}(\mathrm{G})}\frac{1}{\sqrt{\kappa(\acute{p})\times \kappa(\acute{q})}}=\bigg(\frac{m}{m-n+2}\bigg)\sum_{\acute{p}\acute{q} \in \mathrm{E_j}(\mathrm{G})}\sum_{j=1}^9\\&&\frac{1}{\sqrt{\kappa(\acute{p})\times \kappa(\acute{q})}}
\end{eqnarray*}
\begin{eqnarray*}
J(\mathrm{G_1})&=&\bigg(\frac{63n^2-123n+60}{43-84n+42n^2}\bigg)\frac{1}{4}|E_1(\mathrm{G_1})|+\frac{1}{2\sqrt{7}}|E_2(\mathrm{G_1})|+
\frac{1}{2\sqrt{10}}|E_3(\mathrm{G_1})|+\\&&\frac{1}{6\sqrt{2}}|E_4(\mathrm{G_1})|+\frac{1}{\sqrt{70}}|E_5(\mathrm{G_1})|+
\frac{1}{3\sqrt{14}}|E_6(\mathrm{G_1})|+\frac{1}{10}|E_7(\mathrm{G_1})|+\\&&\frac{1}{6\sqrt{5}}|E_8(\mathrm{G_1})|+\frac{1}{18}|E_9(\mathrm{G_1})|
\end{eqnarray*}
By doing some calculation, we get
\begin{eqnarray*}
\Longrightarrow J(\mathrm{G_1})&=&\frac{(20-41n+21n^2)(1595.47+7(-307-270\sqrt{2}+12\sqrt{5}+54\sqrt{10})n)}{70(43-84n+42n^2)}+\\&&
\frac{210(5+3\sqrt{2})n^2}{70(43-84n+42n^2)}.
\end{eqnarray*}
\qed
\end{proof}
Now, we compute $ReZG1$, $ReZG2$ and $ReZG3$ indices of Hex Derived network $\mathrm{G_1}$.
\begin{theorem}
Let $\mathrm{G_1}$ be the Hex Derived network, then\\ \\
$\bullet$ $ReZG1(\mathrm{G_1})$ = $19-39n+21n^2$; \\ \\
$\bullet$ $ReZG2(\mathrm{G_1})$ = $\frac{115452}{425}-\frac{5637n}{11}+\frac{2583n^2}{11}$; \\ \\
$\bullet$ $ReZG3(\mathrm{G_1})$ = $12(27381-38996n+13692n^2)$.
\end{theorem}
\begin{proof}
By using edge partition given in Table 1, the ReZG1$(\mathrm{G_1})$ can be calculated by using $(4)$ as follows.
\begin{eqnarray*}
ReZG1(\mathrm{G})&=&\sum_{\acute{p}\acute{q} \in \mathrm{E}(\mathrm{G})}\bigg(\frac{\kappa(\acute{p})\times \kappa(\acute{q})}{\kappa(\acute{p})+ \kappa(\acute{p})}\bigg)=\sum_{j=1}^9\sum_{\acute{p}\acute{q} \in E_j(\mathrm{G_1})}\bigg(\frac{\kappa(\acute{p})\times \kappa(\acute{q})}{\kappa(\acute{p})+ \kappa(\acute{p})}\bigg)
\end{eqnarray*}

\begin{eqnarray*}
ReZG1(\mathrm{G_1})&=&2|E_1(\mathrm{G_1})|+\frac{28}{11}|E_2(\mathrm{G_1})|+\frac{20}{7}|E_3(\mathrm{G_1})|+
\frac{36}{11}|E_4(\mathrm{G_1})|+\\&&\frac{70}{17}|E_5(\mathrm{G_1})|+\frac{126}{25}|E_6(\mathrm{G_1})|+5|E_7(\mathrm{G_1})|
+\frac{45}{7}|E_8(\mathrm{G_1})|+\\&&9|E_9(\mathrm{G_1})|
\end{eqnarray*}
By doing some calculation, we get\\
$$\Longrightarrow ReZG1(\mathrm{G_1})=19-39n+21n^2$$
The ReZG2$(\mathrm{G_1})$ can be calculated by using $(5)$ as follows.
$$ReZG2(\mathrm{G})=\sum_{\acute{p}\acute{q} \in \mathrm{E}(\mathrm{G})}\bigg(\frac{\kappa(\acute{p})+\kappa(\acute{q})}{\kappa(\acute{p})\times \kappa(\acute{p})}\bigg)=\sum_{j=1}^9\sum_{\acute{p}\acute{q} \in E_j(\mathrm{G_1})}\bigg(\frac{\kappa(\acute{p})+\kappa(\acute{q})}{\kappa(\acute{p})\times \kappa(\acute{p})}\bigg)$$
\begin{eqnarray*}
ReZG2(\mathrm{G_1})&=&\frac{1}{2}|E_1(\mathrm{G_1})|+\frac{11}{28}|E_2(\mathrm{G_1})|+\frac{7}{20}|E_3(\mathrm{G_1})|+
\frac{11}{36}|E_4(\mathrm{G_1})|+\\&&\frac{17}{70}|E_5(\mathrm{G_1})|+\frac{25}{126}|E_6(\mathrm{G_1})|+
\frac{1}{5}|E_7(\mathrm{G_1})|
+\frac{7}{45}|E_8(\mathrm{G_1})|+\\&&\frac{1}{9}|E_9(\mathrm{G_1})|
\end{eqnarray*}
By doing some calculation, we get\\
$$\Longrightarrow ReZG2(\mathrm{G_1})=\frac{115452}{425}-\frac{5637n}{11}+\frac{2583n^2}{11}$$
The ReZG3$(\mathrm{G_1})$ index can be calculated from (6) as follows.
\begin{eqnarray*}
ReZG3(\mathrm{G})&=&\sum_{\acute{p}\acute{q} \in \mathrm{E}(\mathrm{G})}(\kappa(\acute{p})\times\kappa(\acute{q}))(\kappa(\acute{p})+ \kappa(\acute{p})\\&&=\sum_{j=1}^{9}\sum_{\acute{p}\acute{q} \in E_j(\mathrm{G_1})}(\kappa(\acute{p})\times\kappa(\acute{q}))(\kappa(\acute{p})+ \kappa(\acute{p}))
\end{eqnarray*}
\begin{eqnarray*}
ReZG3(\mathrm{G_1})&=&128|E_1(\mathrm{G_1})|+308|E_2(\mathrm{G_1})|+560|E_3(\mathrm{G_1})|+1584|E_4(\mathrm{G_1})|+\\&&1190|E_5(\mathrm{G_1})|+
3150|E_6(\mathrm{G_1})|+2000|E_7(\mathrm{G_1})|
+5040|E_8(\mathrm{G_1})|+\\&&11664|E_9(\mathrm{G_1})|
\end{eqnarray*}
By doing some calculation, we get\\
\begin{eqnarray*}
\Longrightarrow ReZG3(\mathrm{G_1})&=&12(27381-38996n+13692n^2)
\end{eqnarray*}
\end{proof}
\qed
\begin{table}[h]
\centering
\begin{tabular}{|c|c|c|c|c|}
   \hline
  $(\kappa_{x},\kappa_{y})$ where $\acute{p}\acute{q}\in E(\mathrm{G_1})$ & $\text{Number of edges}$&$(\kappa_{u},\kappa_{v})$ where $\acute{p}\acute{q}\in E(\mathrm{G_1})$&$\text{Number of edges}$ \\\hline
   $(25,33)$ & $12$&$(44,44)$ & $18r^2-72r+72$ \\\hline
   $(25,36)$ & $12$&$(44,129)$ & $36$ \\\hline
   $(25,54)$ & $12$&$(44,140)$ & $48r-144$\\\hline
   $(25,77)$ & $12$& $(44,156)$ & $36r^2-180r+228$\\\hline
   $(28,36)$ & $12r-36$& $(54,77)$ & $12$\\\hline
    $(28,77)$ & $12$&$(54,129)$ & $6$ \\\hline
   $(28,80)$ & $12r-48$&$(77,80)$ & $12$ \\\hline
   $(33,36)$ & $12$&$(77,129)$ & $12$\\\hline
   $(33,54)$ & $12$& $(77,140)$ & $12$\\\hline
   $(33,129)$ & $12$& $(80,80)$ & $6r-30$\\\hline
   $(36,36)$ & $12r-30$&$(80,140)$ & $12r-48$ \\\hline
   $(36,44)$ & $12r-24$&$(129,140)$ & $12$\\\hline
   $(36,77)$ & $48$& $(129,156)$ & $6$\\\hline
   $(36,80)$ & $24r-96$& $(140,140)$ & $6r-24$\\\hline
   $(36,129)$ & $24$& $(140,156)$ & $12r-36$\\\hline
   $(36,140)$ & $24r-72$& $(156,156)$ & $9r^2-51r+72$\\\hline
\end{tabular}
\vspace{.2cm}
\caption{Edge partition of Hex Derived network of type 3 $HDN3(r)$ based on degrees of end vertices of each edge.}
\end{table}
Now, we compute $ABC_{4}$ and $GA_{5}$ indices of Hex Derived network $\mathrm{G_1}$.
\begin{theorem}
Let $\mathrm{G_1}$ be the Hex Derived network, then\\ \\
$\bullet$ $ABC_{4}(\mathrm{G_1})$ = $51.706+\frac{3}{20}\sqrt{\frac{79}{2}}(-5+n)+3\sqrt{\frac{53}{70}}(-4+n)+\frac{3}{5} \sqrt{\frac{109}{14}}(-4+n)+\sqrt{\frac{114}{5}}(-4+n)+\frac{3}{35} \sqrt{\frac{139}{2}}(-4+n)+3\sqrt{\frac{14}{65}}(-3+n)+12\sqrt{\frac{26}{55}}(-3+n)+2\sqrt{\frac{174}{35}}(-3+n)+\sqrt{\frac{62}{7}}
(-3+n)+\sqrt{\frac{78}{11}}(-2+n)+\frac{9}{11}\sqrt{\frac{43}{2}}(-2+n)^2+\frac{1}{3}\sqrt{\frac{35}{2}}(-5+2n)+\frac{1}{26} \sqrt{\frac{155}{2}}(24-17n+3n^2)+3\sqrt{\frac{6}{13}}(19-15n+3n^2)$; \\ \\
$\bullet$ $ABC_4(\mathrm{G_1})$ = $110.66+ \frac{6}{37}\sqrt{1365}(-6+n)+\frac{24}{11}\sqrt{7}(-5+n)+\frac{18}{11}\sqrt{35}(-5+n)+\frac{24}{23} \sqrt{385}(-5+n)+\frac{144}{29}\sqrt{5}(-4+n)+\frac{9}{5}\sqrt{11}(-4+n)+\frac{8}{9}\sqrt{35}(-4+n)+\frac{36}{29} \sqrt{22}(-2+n)-12n+3n^2+\frac{3}{2}(42-13n+n^2)+\frac{6}{25}\sqrt{429}(30-11n+n^2)$.
\end{theorem}
\begin{proof}
By using edge partition given in Table 2, the $ABC_{4}(\mathrm{G_1})$ can be calculated by using $(7)$ as follows.
$$ABC_{4}(\mathrm{G})=\sum_{\acute{p}\acute{q} \in \mathrm{E}(\mathrm{G})}\sqrt{\frac{S_{\acute{p}}+S_{\acute{q}}-2}{S_{\acute{p}}S_{\acute{q}}}}=\sum_{j=10}^{41}\sum_{\acute{p}\acute{q} \in E_j(\mathrm{G_1})}\sqrt{\frac{S_{\acute{p}}+S_{\acute{q}}-2}{S_{\acute{p}}S_{\acute{q}}}}$$
\begin{eqnarray*}
ABC_4(\mathrm{G_1})&=&\frac{2}{5}\sqrt{\frac{14}{33}}|E_{10}(\mathrm{G_1})|+\frac{\sqrt{59}}{30}|E_{11}(\mathrm{G_1})|+\frac{1}{15}\sqrt{\frac{77}{6}}|E_{12}(\mathrm{G_1})|+
\\&&\frac{36}{11}\frac{2}{\sqrt{77}}|E_{13}(\mathrm{G_1})|+\frac{1}{6}\sqrt{\frac{31}{14}}|E_{14}(\mathrm{G_1})|+\frac{1}{14}\sqrt{\frac{103}{11}}|E_{15}(\mathrm{G_1})|+
\\&&\frac{1}{4}\sqrt{\frac{53}{70}}|E_{16}(\mathrm{G_1})|+\frac{1}{6}\sqrt{\frac{67}{33}}|E_{17}(\mathrm{G_1})|+\frac{1}{9}\sqrt{\frac{85}{22}}|E_{18}(\mathrm{G_1})|
+\\&&\frac{4}{3}\sqrt{\frac{10}{473}}|E_{19}(\mathrm{G_1})|+\frac{1}{18}\sqrt{\frac{32}{2}}|E_{20}(\mathrm{G_1})|+\frac{1}{2}\sqrt{\frac{13}{66}}|E_{21}(\mathrm{G_1})|
+\\&&\frac{1}{2}\sqrt{\frac{37}{231}}|E_{22}(\mathrm{G_1})|+\frac{1}{4}\sqrt{\frac{19}{30}}|E_{23}(\mathrm{G_1})|+\frac{1}{6}\sqrt{\frac{163}{129}}|E_{24}(\mathrm{G_1})|+
\\&&\frac{1}{2}\sqrt{\frac{29}{210}}|E_{25}(\mathrm{G_1})|+\frac{1}{22}\sqrt{\frac{43}{2}}|E_{26}(\mathrm{G_1})|+\frac{1}{2}\sqrt{\frac{57}{473}}|E_{27}(\mathrm{G_1})|
+\\&&\frac{1}{2}\sqrt{\frac{13}{110}}|E_{28}(\mathrm{G_1})|
+\frac{1}{2}\sqrt{\frac{3}{26}}|E_{29}(\mathrm{G_1})|
+\frac{1}{3}\sqrt{\frac{43}{154}}|E_{30}(\mathrm{G_1})|
+\\&&\frac{1}{9}\sqrt{\frac{181}{86}}|E_{31}(\mathrm{G_1})|
+\frac{1}{4}\sqrt{\frac{31}{77}}|E_{32}(\mathrm{G_1})|
+2\sqrt{\frac{17}{3311}}|E_{33}(\mathrm{G_1})|
+\\&&\frac{1}{14}\sqrt{\frac{43}{11}}|E_{34}(\mathrm{G_1})|
+\frac{1}{40}\sqrt{\frac{79}{2}}|E_{35}(\mathrm{G_1})
+\frac{1}{20}\sqrt{\frac{109}{14}}|E_{36}(\mathrm{G_1})|
+\\&&\frac{1}{2}\sqrt{\frac{99}{1505}}|E_{37}(\mathrm{G_1})
+\frac{1}{6}\sqrt{\frac{283}{559}}|E_{38}(\mathrm{G_1})
+\frac{1}{70}\sqrt{\frac{139}{2}}|E_39(\mathrm{G_1})
+\\&&\frac{1}{2}\sqrt{\frac{7}{130}}|E_{40}(\mathrm{G_1})
+\frac{1}{78}\sqrt{\frac{155}{2}}|E_{41}(\mathrm{G_1})
\end{eqnarray*}
By doing some calculation, we get\\
\begin{eqnarray*}
\Longrightarrow ABC_4(\mathrm{G_1})&=&51.706+\frac{3}{20}\sqrt{\frac{79}{2}}(-5+n)+3\sqrt{\frac{53}{70}}(-4+n)+\frac{3}{5} \sqrt{\frac{109}{14}}(-4+n)+\\&&\sqrt{\frac{114}{5}}(-4+n)+\frac{3}{35} \sqrt{\frac{139}{2}}(-4+n)+3\sqrt{\frac{14}{65}}(-3+n)+\\&&12\sqrt{\frac{26}{55}}(-3+n)+2\sqrt{\frac{174}{35}}(-3+n)+\sqrt{\frac{62}{7}}
(-3+n)+\\&&\sqrt{\frac{78}{11}}(-2+n)+\frac{9}{11}\sqrt{\frac{43}{2}}(-2+n)^2+\frac{1}{3}\sqrt{\frac{35}{2}}(-5+2n)+\frac{1}{26} \sqrt{\frac{155}{2}}\\&&(24-17n+3n^2)+3\sqrt{\frac{6}{13}}(19-15n+3n^2)
\end{eqnarray*}
The $GA_{5}(\mathrm{G_1})$ index can be calculated from (8) as follows.
\begin{eqnarray*}
GA_{5}(\mathrm{G})=\sum_{\acute{p}\acute{q} \in \mathrm{E}(\mathrm{G})}\frac{2\sqrt{S_{\acute{p}}S_{\acute{q}}}}{(S_{\acute{p}}+S_{\acute{q}})}=\sum_{j=1}^{41}\sum_{\acute{p}\acute{q} \in E_j(\mathrm{G_1})}\frac{2\sqrt{S_{\acute{p}}S_{\acute{q}}}}{(S_{\acute{p}}+S_{\acute{q}})}
\end{eqnarray*}
\begin{eqnarray*}
GA_5(\mathrm{G_1})&=&\frac{5}{29}\sqrt{33}|E_{10}(\mathrm{G_1})|+\frac{60}{11}|E_{11}(\mathrm{G_1})|+
\frac{30}{79}\sqrt{6}|E_{12}(\mathrm{G_1})|+
\\&&\frac{5}{51}\sqrt{77}|E_{13}(\mathrm{G_1})|+\frac{3}{8}\sqrt{7}|E_{14}(\mathrm{G_1})|+
\frac{4}{15}\sqrt{11}|E_{15}(\mathrm{G_1})|+
\\&&\frac{4}{27}\sqrt{35}|E_{16}(\mathrm{G_1})|+\frac{4}{23}\sqrt{33}|E_{17}(\mathrm{G_1})|+
\frac{6}{29}\sqrt{22}|E_{18}(\mathrm{G_1})|
+\\&&\frac{1}{31}\sqrt{957}|E_{19}(\mathrm{G_1})|+|E_{20}(\mathrm{G_1})|+
\frac{3}{10}\sqrt{11}|E_{21}(\mathrm{G_1})|
+\\&&\frac{12}{113}\sqrt{77}|E_{22}(\mathrm{G_1})|+\frac{12}{29}\sqrt{5}|E_{23}(\mathrm{G_1})|+
\frac{4}{55}\sqrt{129}|E_{24}(\mathrm{G_1})|+
\\&&\frac{3}{22}\sqrt{35}|E_{25}(\mathrm{G_1})|+|E_{26}(\mathrm{G_1})|+
\frac{4}{173}\sqrt{1419}|E_{27}(\mathrm{G_1})|
+\\&&\frac{1}{23}\sqrt{385}|E_{28}(\mathrm{G_1})|
+\frac{1}{25}\sqrt{429}|E_{29}(\mathrm{G_1})|
+\frac{6}{131}\sqrt{462}|E_{30}(\mathrm{G_1})|
+\\&&\frac{6}{61}\sqrt{86}|E_{31}(\mathrm{G_1})|
+\frac{8}{157}\sqrt{385}|E_{32}(\mathrm{G_1})|
+\frac{1}{103}\sqrt{9933}|E_{33}(\mathrm{G_1})|
+\\&&\frac{4}{31}\sqrt{55}|E_{34}(\mathrm{G_1})|
+|E_{35}(\mathrm{G_1})
+\frac{4}{11}\sqrt{7}|E_{36}(\mathrm{G_1})|
+\\&&\frac{4}{269}\sqrt{4515}|E_{37}(\mathrm{G_1})|
+\frac{4}{95}\sqrt{559}|E_{38}(\mathrm{G_1})|
+|E_39(\mathrm{G_1})|
+\\&&\frac{1}{37}\sqrt{1365}|E_{40}(\mathrm{G_1})|
+|E_{41}(\mathrm{G_1})|
\end{eqnarray*}
By doing some calculation, we get\\
\begin{eqnarray*}
\Longrightarrow GA_5(\mathrm{G_1})&=& 315.338+\frac{288}{29}\sqrt{5}(-4+r)+\frac{48}{11}\sqrt{7}(-4+r)+\frac{16}{9}\sqrt{35}(-4+r)+\frac{9}{2}\\&& \sqrt{7}(-3+r)+\frac{36}{11}\sqrt{35}(-3+r)+\frac{48}{23}\sqrt{385}(-3+r)+\frac{12}{37}\sqrt{1365}(-3+r)+\\&&\frac{18}{5}\sqrt{11}(-2+r)-99r+27r^2+\frac{12}{25} \sqrt{429}(19-15r+3r^2)
\end{eqnarray*}
\end{proof}
 \subsection{Results for Third type of Rectangular Hex Derived network $THDN3(r)$}
 Now, we discuss the newly derived third type of rectangular hex derived network and compute numerical and exact results for forgotten index, Balaban index, reclassified the Zagreb indices, forth version of $ABC$ index and fifth version of $GA$ index for $THDN3(r)$.
\begin{theorem}
Consider the Third type of Triangular Hex Derived network of $THDN3(n)$, then its forgotten index is equal to
\begin{equation*}
F(HDN3(n))=12(990-997r+259r^2)
\end{equation*}
\end{theorem}
\begin{figure}
\centering
  \includegraphics[width=5cm]{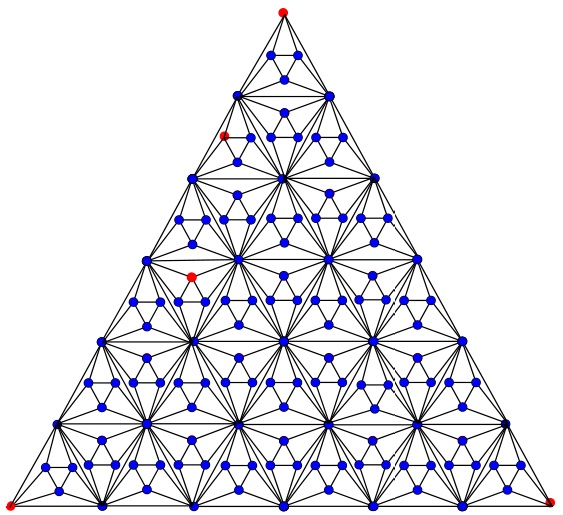}\\
  \caption{Hex Derived network of type 3 ($HDN3(4)$)}\label{fig 2}
\end{figure}
\vspace{1cm}
\begin{table}
\centering
\begin{tabular}{|c|c|c|c|c|}
   \hline
  $(\kappa_{x},\kappa_{y})$ where $\acute{p}\acute{q}\in E(\mathrm{G_1})$ & $\text{Number of edges}$&$(\kappa_{u},\kappa_{v})$ where $\acute{p}\acute{q}\in E(\mathrm{G_1})$&$\text{Number of edges}$ \\\hline
   $(4,4)$ & $3r^2-6r+9$&$(10,10)$ & $3r-6$ \\\hline
   $(4,10)$ & $18r-30$&$(10,18)$ & $6r-18$ \\\hline
   $(4,18)$ & $6r^2-30r+36$&$(18,18)$ & $\frac{3r^2-21r+36}{2}$\\\hline
\end{tabular}
\vspace{.2cm}
\caption{Edge partition Third type of Triangular Hex Derived network $THDN3(r)$ based on degrees of end vertices of each edge.}
\end{table}
\vspace{-1cm}
\begin{proof}
Let $\mathrm{G_2}$ be the Hex Derived network of type 3, $HDN3(r)$ shown in Fig. \ref{fig 2}, where $r\geq4$. The Hex Derived network $\mathrm{G_2}$ has $\frac{7r^2-11r+6}{2}$ vertices and the edge set of $\mathrm{G_2}$ is divided into six partitions based on the degree of end vertices. The first edge partition $E_1(\mathrm{G_2})$ contains $3r^2-6r+9$ edges $\acute{p}\acute{q}$, where $\kappa(\acute{p})=\kappa(\acute{q})=4$. The second edge partition $E_2(\mathrm{G_2})$ contains $18r-30$ edges $\acute{p}\acute{q}$, where $\kappa(\acute{p})=4$ and $\kappa(\acute{q})=10$. The third edge partition $E_3(\mathrm{G_2})$ contains $6r^2-30r+36$ edges $\acute{p}\acute{q}$, where $\kappa(\acute{p})=4$ and $\kappa(\acute{q})=18$. The fourth edge partition $E_4(\mathrm{G_2})$ contains $3r-6$ edges $\acute{p}\acute{q}$, where $\kappa(\acute{p})=\kappa(\acute{q})=10$. The fifth edge partition $E_5(\mathrm{G_2})$ contains $6r-18$ edges $\acute{p}\acute{q}$, where $\kappa(\acute{p})=10$ and $\kappa(\acute{q})=18$ and the sixth edge partition $E_6(\mathrm{G_2})$ contains $\frac{3r^2-21r+36}{2}$ edges $\acute{p}\acute{q}$, where $\kappa(\acute{p})=\kappa(\acute{q})=18$. Table 3, shows such an edge partition of $\mathrm{G_2}$. Thus from $(3)$ this follows that
\begin{equation*}
F(G)=\sum_{\acute{p}\acute{q} \in \mathrm{E}(\mathrm{G})}((\kappa(\acute{p}))^2+(\kappa(\acute{q}))^2)
\end{equation*}
Let $\mathrm{G_2}$ be the third type of triangular hex derived network, $THDN3(r)$. By using edge partition from Table 3, the result follows. The forgotten index can be calculated by using (2) as follows.
\begin{equation*}
  F(\mathrm{G_2})=\sum_{\acute{p}\acute{q} \in \mathrm{E}(\mathrm{G})}((\kappa(\acute{p}))^2+(\kappa(\acute{q}))^2)=\sum_{\acute{p}\acute{q} \in \mathrm{E_{j}}(\mathrm{G})}\sum_{j=1}^6((\kappa(\acute{p}))^2+(\kappa(\acute{q}))^2)
\end{equation*}
\begin{eqnarray*}
F(\mathrm{G_2})&=&32|E_1(\mathrm{G_2})|+116|E_2(\mathrm{G_2})|+340|E_3(\mathrm{G_2})|+200|E_4(\mathrm{G_2})|+424|E_5(\mathrm{G_2})|+
\\&&648|E_6(\mathrm{G_2})|
\end{eqnarray*}
\end{proof}
By doing some calculations, we get\\
$$\Longrightarrow F(\mathrm{G_2})=12(990-997r+259r^2).$$
\qed

In the following theorem, we compute Balaban index of third type of triangular hex Derived network, $\mathrm{G_2}$.
\begin{theorem}
For triangular hex derived network $\mathrm{G_2}$, the Balaban index is equal to
\begin{eqnarray*}
J(\mathrm{G_2})&=&\bigg(\frac{1}{40(8-14r+7r^2)}\bigg)(6-13r+7r^2)(159+1802\sqrt{2}-36\sqrt{5}-\\&&90\sqrt{10}+(-107-150\sqrt{2}+12\sqrt{5}+54\sqrt{10})r+
10(5+3\sqrt{2})r^2)
\end{eqnarray*}
\end{theorem}
\begin{proof}
Let $\mathrm{G_2}$ be the triangular hex derived network $THDN3(r)$. By using edge partition from Table 3, the result follows. The Balaban index can be calculated by using (3) as follows.
\begin{eqnarray*}
 J(\mathrm{G_2})&=&\bigg(\frac{m}{m-n+2}\bigg)\sum_{\acute{p}\acute{q} \in \mathrm{E}(\mathrm{G})}\frac{1}{\sqrt{\kappa(\acute{p})\times \kappa(\acute{q})}}=\bigg(\frac{m}{m-n+2}\bigg)\sum_{\acute{p}\acute{q} \in \mathrm{E_j}(\mathrm{G})}\sum_{j=1}^6\\&&\frac{1}{\sqrt{\kappa(\acute{p})\times \kappa(\acute{q})}}
\end{eqnarray*}
\begin{eqnarray*}
J(\mathrm{G_2})&=&\frac{3}{2}\bigg(\frac{6-13r+7r^2}{8-14r+7r^2}\bigg)\frac{1}{4}|E_1(\mathrm{G_2})|+
\frac{1}{2\sqrt{10}}|E_2(\mathrm{G_2})|+\frac{1}{6\sqrt{2}}|E_3(\mathrm{G_2})|+\\&&\frac{1}{10}|E_4(\mathrm{G_2})|+
\frac{1}{6\sqrt{5}}|E_5(\mathrm{G_2})|+\frac{1}{18}|E_6(\mathrm{G_2})|
\end{eqnarray*}
By doing some calculation, we get
\begin{eqnarray*}
\Longrightarrow J(\mathrm{G_2})&=&\bigg(\frac{1}{40(8-14r+7r^2)}\bigg)(6-13r+7r^2)(159+1802\sqrt{2}-36\sqrt{5}-\\&&90\sqrt{10}+(-107-150\sqrt{2}+12\sqrt{5}+54\sqrt{10})r+
10(5+3\sqrt{2})r^2).
\end{eqnarray*}
\qed
\end{proof}
Now, we compute $ReZG1$, $ReZG2$ and $ReZG3$ indices of triangular hex derived network $\mathrm{G_2}$.
\begin{theorem}
Let $\mathrm{G_2}$ be the triangular hex derived network, then\\ \\
$\bullet$ $ReZG1(\mathrm{G_2})$ = $\frac{3}{154}(3408-5117r+2009r^2)$; \\ \\
$\bullet$ $ReZG2(\mathrm{G_2})$ = $\frac{1}{2}(6-11r+7r^2)$; \\ \\
$\bullet$ $ReZG3(\mathrm{G_2})$ = $24(6192-5185r+1141r^2)$.
\end{theorem}
\begin{proof}
By using edge partition given in Table 3, the ReZG1$(\mathrm{G_2})$ can be calculated by using $(4)$ as follows.
\begin{eqnarray*}
ReZG1(\mathrm{G})&=&\sum_{\acute{p}\acute{q} \in \mathrm{E}(\mathrm{G})}\bigg(\frac{\kappa(\acute{p})\times \kappa(\acute{q})}{\kappa(\acute{p})+ \kappa(\acute{p})}\bigg)=\sum_{j=1}^6\sum_{\acute{p}\acute{q} \in E_j(\mathrm{G_2})}\bigg(\frac{\kappa(\acute{p})\times \kappa(\acute{q})}{\kappa(\acute{p})+ \kappa(\acute{p})}\bigg)
\end{eqnarray*}

\begin{eqnarray*}
ReZG1(\mathrm{G_2})&=&2|E_1(\mathrm{G_2})|+\frac{20}{7}|E_2(\mathrm{G_2})|+
\frac{36}{11}|E_3(\mathrm{G_2})|+5|E_4(\mathrm{G_2})|
+\\&&\frac{45}{7}|E_5(\mathrm{G_2})|+9|E_6(\mathrm{G_2})|
\end{eqnarray*}
By doing some calculation, we get\\
$$\Longrightarrow ReZG1(\mathrm{G_2})=\frac{3}{154}(3408-5117r+2009r^2)$$
The ReZG2$(\mathrm{G_2})$ can be calculated by using $(5)$ as follows.
$$ReZG2(\mathrm{G})=\sum_{\acute{p}\acute{q} \in \mathrm{E}(\mathrm{G})}\bigg(\frac{\kappa(\acute{p})+\kappa(\acute{q})}{\kappa(\acute{p})\times \kappa(\acute{p})}\bigg)=\sum_{j=1}^9\sum_{\acute{p}\acute{q} \in E_j(\mathrm{G_2})}\bigg(\frac{\kappa(\acute{p})+\kappa(\acute{q})}{\kappa(\acute{p})\times \kappa(\acute{p})}\bigg)$$
\begin{eqnarray*}
ReZG2(\mathrm{G_2})&=&\frac{1}{2}|E_1(\mathrm{G_2})|+\frac{7}{20}|E_2(\mathrm{G_2})|+
\frac{11}{36}|E_3(\mathrm{G_2})|+
\frac{1}{5}|E_4(\mathrm{G_2})|
+\\&&\frac{7}{45}|E_5(\mathrm{G_2})|+\frac{1}{9}|E_6(\mathrm{G_2})|
\end{eqnarray*}
By doing some calculation, we get\\
$$\Longrightarrow ReZG2(\mathrm{G_2})=\frac{1}{2}(6-11r+7r^2)$$
The ReZG3$(\mathrm{G_2})$ index can be calculated from (6) as follows.
\begin{eqnarray*}
ReZG3(\mathrm{G})&=&\sum_{\acute{p}\acute{q} \in \mathrm{E}(\mathrm{G})}(\kappa(\acute{p})\times\kappa(\acute{q}))(\kappa(\acute{p})+ \kappa(\acute{p})\\&&=\sum_{j=1}^{6}\sum_{\acute{p}\acute{q} \in E_j(\mathrm{G_2})}(\kappa(\acute{p})\times\kappa(\acute{q}))(\kappa(\acute{p})+ \kappa(\acute{p}))
\end{eqnarray*}
\begin{eqnarray*}
ReZG3(\mathrm{G_2})&=&128|E_1(\mathrm{G_2})|+560|E_2(\mathrm{G_2})|+1584|E_3(\mathrm{G_2})|+2000|E_4(\mathrm{G_2})|
+\\&&5040|E_5(\mathrm{G_2})|+11664|E_6(\mathrm{G_2})|
\end{eqnarray*}
By doing some calculation, we get\\
\begin{eqnarray*}
\Longrightarrow ReZG3(\mathrm{G_2})&=&24(6192-5185r+1141r^2)
\end{eqnarray*}
\end{proof}
\qed
\begin{table}[h]
\centering
\begin{tabular}{|c|c|c|c|c|}
   \hline
  $(\kappa_{x},\kappa_{y})$ where $\acute{p}\acute{q}\in E(\mathrm{G_2})$ & $\text{Number of edges}$&$(\kappa_{u},\kappa_{v})$ where $\acute{p}\acute{q}\in E(\mathrm{G_2})$&$\text{Number of edges}$ \\\hline
   $(22,22)$ & $3$&$(44,124)$ & $12$ \\\hline
   $(22,28)$ & $12$&$(44,140)$ & $24r-120$ \\\hline
   $(22,36)$ & $6$&$(44,156)$ & $6r^2-66r+180$\\\hline
   $(22,66)$ & $6r-12$& $(66,66)$ & $3$\\\hline
   $(28,66)$ & $24$& $(66,80)$ & $6$\\\hline
    $(28,80)$ & $6r-24$&$(66,124)$ & $6$ \\\hline
    $(36,36)$ & $6r-18$&$(80,80)$ & $3r-15$ \\\hline
   $(36,44)$ & $6r-24$&$(80,124)$ & $6$ \\\hline
   $(36,66)$ & $12$&$(80,140)$ & $6r-30$\\\hline
   $(36,80)$ & $12r-48$& $(124,140)$ & $6$\\\hline
   $(36,124)$ & $24$& $(140,140)$ & $3r-15$\\\hline
   $(36,140)$ & $12r-60$&$(140,156)$ & $6r-36$ \\\hline
   $(44,44)$ & $3r^2-24r+48$&$(156,156)$ & $\frac{3r^2-39r+126}{2}$\\\hline
\end{tabular}
\vspace{.2cm}
\caption{Edge partition of Hex Derived network of type 3 $HDN3(r)$ based on degrees of end vertices of each edge.}
\end{table}
Now, we compute $ABC_{4}$ and $GA_{5}$ indices of triangular hex derived network $\mathrm{G_2}$.
\begin{theorem}
Let $\mathrm{G_2}$ be the triangular hex derived network, then\\ \\
$\bullet$ $ABC_{4}(\mathrm{G_2})$ = $24.131+3\sqrt{\frac{7}{130}}(-6+r)+6\sqrt{\frac{26}{55}}(-5+r)+\sqrt{\frac{174}{35}}(-5+r)+\frac{3}{10} \sqrt{\frac{109}{14}}(-5+r)+\frac{3}{40}\sqrt{\frac{79}{2}}(-5+r)+\frac{3}{70}\sqrt{\frac{139}{2}}(-5+r)+\frac{3}{2} \sqrt{\frac{53}{70}}(-4+r)+\sqrt{\frac{39}{22}}(-4+r)+\sqrt{\frac{57}{10}}(-4+r)+\frac{3}{22}\sqrt{\frac{43}{2}}(-4+r)^2+\frac{1}{3} \sqrt{\frac{35}{2}}(-3+r)+2\sqrt{\frac{7}{11}}(-2+r)+\frac{1}{52}\sqrt{\frac{155}{2}}(42-13r+r^2)+3\sqrt{\frac{3}{26}}(30-11r+r^2)$; \\ \\
$\bullet$ $GA_5(\mathrm{G_2})$ = $110.66+ \frac{6}{37}\sqrt{1365}(-6+r)+\frac{24}{11}\sqrt{7}(-5+r)+\frac{18}{11}\sqrt{35}(-5+r)+\frac{24}{23} \sqrt{385}(-5+r)+\frac{144}{29}\sqrt{5}(-4+r)+\frac{9}{5}\sqrt{11}(-4+r)+\frac{8}{9}\sqrt{35}(-4+r)+\frac{36}{29} \sqrt{22}(-2+r)-12r+3r^2+\frac{3}{2}(42-13r+r^2)+\frac{6}{25}\sqrt{429}(30-11r+r^2)$.
\end{theorem}
\begin{proof}
By using edge partition given in Table 4, the $ABC_{4}(\mathrm{G_2})$ can be calculated by using $(7)$ as follows.
$$ABC_{4}(\mathrm{G})=\sum_{\acute{p}\acute{q} \in \mathrm{E}(\mathrm{G})}\sqrt{\frac{S_{\acute{p}}+S_{\acute{q}}-2}{S_{\acute{p}}S_{\acute{q}}}}=\sum_{j=7}^{32}\sum_{\acute{p}\acute{q} \in E_j(\mathrm{G_2})}\sqrt{\frac{S_{\acute{p}}+S_{\acute{q}}-2}{S_{\acute{p}}S_{\acute{q}}}}$$
\begin{eqnarray*}
ABC_4(\mathrm{G_2})&=&\frac{1}{11}\sqrt{\frac{21}{2}}|E_{7}(\mathrm{G_2})|+\sqrt{\frac{6}{77}}|E_{8}(\mathrm{G_2})|+
\frac{1}{3}\sqrt{\frac{7}{11}}|E_{9}(\mathrm{G_2})|+\frac{1}{11}\sqrt{\frac{43}{6}}|E_{10}(\mathrm{G_2})|+\\&&
\sqrt{\frac{23}{462}}|E_{11}(\mathrm{G_2})|+\frac{1}{4}\sqrt{\frac{53}{70}}|E_{12}(\mathrm{G_2})|+
\frac{1}{18}\sqrt{\frac{32}{2}}|E_{13}(\mathrm{G_2})|+\\&&\frac{1}{2}\sqrt{\frac{13}{66}}|E_{14}(\mathrm{G_2})|+
\frac{5}{3}\frac{1}{\sqrt{6}}|E_{15}(\mathrm{G_2})|+
\frac{1}{4}\sqrt{\frac{19}{30}}|E_{16}(\mathrm{G_2})|+\frac{1}{6}\sqrt{\frac{79}{62}}|E_{17}(\mathrm{G_2})|+\\&&
\frac{1}{2}\sqrt{\frac{29}{210}}|E_{18}(\mathrm{G_2})|
+\frac{1}{22}\sqrt{\frac{43}{2}}|E_{19}(\mathrm{G_2})|+\frac{1}{2}\sqrt{\frac{83}{682}}|E_{20}(\mathrm{G_2})|+\\&&
\frac{1}{2}\sqrt{\frac{13}{110}}|E_{21}(\mathrm{G_2})|
+\frac{1}{2}\sqrt{\frac{3}{26}}|E_{22}(\mathrm{G_2})|+\frac{1}{33}\sqrt{\frac{65}{2}}|E_{23}(\mathrm{G_2})|+\\&&
\sqrt{\frac{3}{110}}|E_{24}(\mathrm{G_2})|+\sqrt{\frac{47}{2046}}|E_{25}(\mathrm{G_2})|
+\frac{1}{40}\sqrt{\frac{79}{2}}|E_{26}(\mathrm{G_2})|
+\\&&\frac{1}{4}\sqrt{\frac{101}{310}}|E_{27}(\mathrm{G_2})|
+\frac{1}{20}\sqrt{\frac{109}{14}}|E_{28}(\mathrm{G_2})|
+\frac{1}{2}\sqrt{\frac{131}{2170}}|E_{29}(\mathrm{G_2})|
+\\&&\frac{1}{70}\sqrt{\frac{139}{2}}|E_{30}(\mathrm{G_2})|
+\frac{1}{2}\sqrt{\frac{7}{130}}|E_{31}(\mathrm{G_2})|
+\frac{1}{78}\sqrt{\frac{155}{2}}|E_{32}(\mathrm{G_2})
\end{eqnarray*}
By doing some calculation, we get\\
\begin{eqnarray*}
\Longrightarrow ABC_4(\mathrm{G_2})&=&24.131+3\sqrt{\frac{7}{130}}(-6+r)+6\sqrt{\frac{26}{55}}(-5+r)+\sqrt{\frac{174}{35}}(-5+r)+\\&&\frac{3}{10} \sqrt{\frac{109}{14}}(-5+r)+\frac{3}{40}\sqrt{\frac{79}{2}}(-5+r)+\frac{3}{70}\sqrt{\frac{139}{2}}(-5+r)+\\&&\frac{3}{2} \sqrt{\frac{53}{70}}(-4+r)+\sqrt{\frac{39}{22}}(-4+r)+\sqrt{\frac{57}{10}}(-4+r)+\frac{3}{22}\\&&\sqrt{\frac{43}{2}}(-4+r)^2+\frac{1}{3} \sqrt{\frac{35}{2}}(-3+r)+2\sqrt{\frac{7}{11}}(-2+r)+\frac{1}{52}\\&&\sqrt{\frac{155}{2}}(42-13r+r^2)+3\sqrt{\frac{3}{26}}(30-11r+r^2)
\end{eqnarray*}
The $GA_{5}(\mathrm{G_2})$ index can be calculated from (8) as follows.
\begin{eqnarray*}
GA_{5}(\mathrm{G})=\sum_{\acute{p}\acute{q} \in \mathrm{E}(\mathrm{G})}\frac{2\sqrt{S_{\acute{p}}S_{\acute{q}}}}{(S_{\acute{p}}+S_{\acute{q}})}=\sum_{j=7}^{32}\sum_{\acute{p}\acute{q} \in E_j(\mathrm{G_2})}\frac{2\sqrt{S_{\acute{p}}S_{\acute{q}}}}{(S_{\acute{p}}+S_{\acute{q}})}
\end{eqnarray*}
\begin{eqnarray*}
GA_5(\mathrm{G_2})&=&1|E_{7}(\mathrm{G_2})|+\frac{2}{25}\sqrt{154}|E_{8}(\mathrm{G_2})|+
\frac{6}{29}\sqrt{22}|E_{9}(\mathrm{G_2})|+\frac{1}{2}\sqrt{3}|E_{10}(\mathrm{G_2})|+\\&&\frac{2}{47}\sqrt{462}|E_{11}(\mathrm{G_2})|+
\frac{4}{27}\sqrt{35}|E_{12}(\mathrm{G_2})|+
1|E_{13}(\mathrm{G_2})|+\\&&\frac{3}{10}\sqrt{11}|E_{14}(\mathrm{G_2})|+
\frac{2}{17}\sqrt{66}|E_{15}(\mathrm{G_2})|+
\frac{12}{29}\sqrt{5}|E_{16}(\mathrm{G_2})|+\\&&\frac{3}{20}\sqrt{31}|E_{17}(\mathrm{G_2})|+
\frac{3}{22}\sqrt{35}|E_{18}(\mathrm{G_2})|
+1|E_{19}(\mathrm{G_2})|+\\&&\frac{1}{21}\sqrt{341}|E_{20}(\mathrm{G_2})|+
\frac{1}{23}\sqrt{385}|E_{21}(\mathrm{G_2})|
+\frac{1}{25}\sqrt{429}|E_{22}(\mathrm{G_2})|+\\&&1|E_{23}(\mathrm{G_2})|+
\frac{4}{73}\sqrt{330}|E_{24}(\mathrm{G_2})|+
\frac{2}{95}\sqrt{2046}|E_{25}(\mathrm{G_2})|+1|E_{26}(\mathrm{G_2})|+
\\&&\frac{4}{51}\sqrt{155}|E_{27}(\mathrm{G_2})|
+\frac{4}{11}\sqrt{7}|E_{28}(\mathrm{G_2})|
+\frac{1}{33}\sqrt{1085}|E_{29}(\mathrm{G_2})|
+\\&&1|E_{30}(\mathrm{G_2})|
+\frac{1}{37}\sqrt{1365}|E_{31}(\mathrm{G_2})|
+1|E_{32}(\mathrm{G_2})|
\end{eqnarray*}
By doing some calculation, we get\\
\begin{eqnarray*}
\Longrightarrow GA_5(\mathrm{G_2})&=& 315.338+\frac{288}{29}\sqrt{5}(-4+r)+\frac{48}{11}\sqrt{7}(-4+r)+\frac{16}{9}\sqrt{35}(-4+r)+\frac{9}{2}\\&& \sqrt{7}(-3+r)+\frac{36}{11}\sqrt{35}(-3+r)+\frac{48}{23}\sqrt{385}(-3+r)+\frac{12}{37}\sqrt{1365}(-3+r)+\\&&\frac{18}{5}\sqrt{11}(-2+r)-99r+27r^2+\frac{12}{25} \sqrt{429}(19-15r+3r^2)
\end{eqnarray*}
\end{proof}
\subsection{Results for Third type of Rectangular Hex Derived netwrok $RHDN3(r)$}
Now, we calculate certain degree based topological indices of Rectangular Hex Derived network of type 3, $RHDN3(r,s)$ of dimension $r=s$. We compute forgotten index, Balaban index, reclassified the Zagreb indices, forth version of $ABC$ index and fifth version of $GA$ index in the coming theorems of $RHDN3(r,s)$.
\begin{theorem}
Consider the third type of rectangular hex derived network $RHDN3(n)$, then its forgotten index is equal to
\begin{equation*}
F(HDN3(n))=19726-20096r+6216r^2
\end{equation*}
\end{theorem}
\begin{figure}
\centering
  \includegraphics[width=5cm]{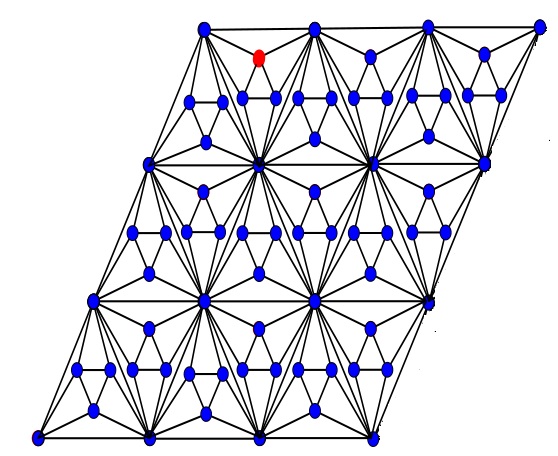}\\
  \caption{Hex Derived network of type 3 ($HDN3(4)$)}\label{fig 2}
\end{figure}
\vspace{1cm}
\begin{table}
\centering
\begin{tabular}{|c|c|c|c|c|}
   \hline
  $(\kappa_{x},\kappa_{y})$ where $\acute{p}\acute{q}\in E(\mathrm{G_1})$ & $\text{Number of edges}$&$(\kappa_{u},\kappa_{v})$ where $\acute{p}\acute{q}\in E(\mathrm{G_1})$&$\text{Number of edges}$ \\\hline
   $(4,4)$ & $6r^2-12r+10$&$(7,18)$ & $2$ \\\hline
   $(4,7)$ & $8$&$(10,10)$ & $4r-10$ \\\hline
   $(4,10)$ & $24r-44$&$(10,18)$ & $8r-20$\\\hline
   $(4,18)$ & $12r^2-48r+48$& $(18,18)$ & $3r^2-16r+21$\\\hline
   $(7,10)$ & $4$&&\\\hline
\end{tabular}
\vspace{.2cm}
\caption{Edge partition of Rectangular Hex Derived network of type 3, $RHDN3(r)$ based on degrees of end vertices of each edge.}
\end{table}
\vspace{-1cm}
\begin{proof}
Let $\mathrm{G_3}$ be the Rectangular Hex Derived network of type 3, $RHDN3(r)$ shown in Fig. \ref{fig 3}, where $r=s\geq4$. The Rectangular Hex Derived network $\mathrm{G_3}$ has $7r^2-12r+6$ vertices and the edge set of $\mathrm{G_3}$ is divided into nine partitions based on the degree of end vertices. The first edge partition $E_1(\mathrm{G_3})$ contains $6r^2-12r+10$ edges $\acute{p}\acute{q}$, where $\kappa(\acute{p})=\kappa(\acute{q})=4$. The second edge partition $E_2(\mathrm{G_1})$ contains $8$ edges $\acute{p}\acute{q}$, where $\kappa(\acute{p})=4$ and $\kappa(\acute{q})=7$. The third edge partition $E_3(\mathrm{G_3})$ contains $24r-44$ edges $\acute{p}\acute{q}$, where $\kappa(\acute{p})=4$ and $\kappa(\acute{q})=10$. The fourth edge partition $E_4(\mathrm{G_3})$ contains $12r^2-48r+48$ edges $\acute{p}\acute{q}$, where $\kappa(\acute{p})=4$ and $\kappa(\acute{q})=18$. The fifth edge partition $E_5(\mathrm{G_3})$ contains $4$ edges $\acute{p}\acute{q}$, where $\kappa(\acute{p})=7$ and $\kappa(\acute{q})=10$. The sixth edge partition $E_6(\mathrm{G_3})$ contains $2$ edges $\acute{p}\acute{q}$, where $\kappa(\acute{p})=7$ and $\kappa(\acute{q})=18$. The seventh edge partition $E_7(\mathrm{G_3})$ contains $4r-10$ edges $\acute{p}\acute{q}$, where $\kappa(\acute{p})=\kappa(\acute{q})=10$ and the eighth edge partition $E_8(\mathrm{G_3})$ contains $8r-20$ edges $\acute{p}\acute{q}$, where $\kappa(\acute{p})=10$ and $\kappa(\acute{q})=18$ and the ninth edge partition $E_9(\mathrm{G_3})$ contains $3r^2-16r+21$ edges $\acute{p}\acute{q}$, where $\kappa(\acute{p})=\kappa(\acute{q})=18$, Table 3 shows such an edge partition of $\mathrm{G_3}$. Thus from $(3)$ this follows that
\begin{equation*}
F(G)=\sum_{\acute{p}\acute{q} \in \mathrm{E}(\mathrm{G})}((\kappa(\acute{p}))^2+(\kappa(\acute{q}))^2)
\end{equation*}
Let $\mathrm{G_3}$ be the third type of triangular hex derived network, $THDN3(r)$. By using edge partition from Table 3, the result follows. The forgotten index can be calculated by using (2) as follows.
\begin{equation*}
  F(\mathrm{G_3})=\sum_{\acute{p}\acute{q} \in \mathrm{E}(\mathrm{G})}((\kappa(\acute{p}))^2+(\kappa(\acute{q}))^2)=\sum_{\acute{p}\acute{q} \in \mathrm{E_{j}}(\mathrm{G})}\sum_{j=1}^9((\kappa(\acute{p}))^2+(\kappa(\acute{q}))^2)
\end{equation*}
\begin{eqnarray*}
F(\mathrm{G_3})&=&32|E_1(\mathrm{G_3})|+65|E_2(\mathrm{G_3})|+116|E_3(\mathrm{G_3})|+340|E_4(\mathrm{G_3})|+
149|E_5(\mathrm{G_3})|+\\&&373|E_6(\mathrm{G_3})|+200|E_7(\mathrm{G_3})|+424|E_8(\mathrm{G_3})|+648|E_9(\mathrm{G_3})|
\end{eqnarray*}
\end{proof}
By doing some calculations, we get\\
$$\Longrightarrow F(\mathrm{G_3})=19726-20096r+6216r^2.$$
\qed

In the following theorem, we compute Balaban index of third type of triangular hex Derived network, $\mathrm{G_3}$.
\begin{theorem}
For triangular hex derived network $\mathrm{G_3}$, the Balaban index is equal to
\begin{eqnarray*}
J(\mathrm{G_3})&=&\bigg(\frac{1}{315(15-28r+14r^2)}\bigg)7(-157-180\sqrt{2}+12\sqrt{5}+54\sqrt{10})r+\\&&105(5+3\sqrt{2})r^2)
(19-40r+21r^2)(3(280+420\sqrt{2}-70\sqrt{5}+\\&&60\sqrt{7}-231\sqrt{10}+5\sqrt{14}+6\sqrt{70}))
\end{eqnarray*}
\end{theorem}
\begin{proof}
Let $\mathrm{G_3}$ be the rectangular hex derived network $RHDN3(r)$. By using edge partition from Table 5, the result follows. The Balaban index can be calculated by using (3) as follows.
\begin{eqnarray*}
 J(\mathrm{G_3})&=&\bigg(\frac{m}{m-n+2}\bigg)\sum_{\acute{p}\acute{q} \in \mathrm{E}(\mathrm{G})}\frac{1}{\sqrt{\kappa(\acute{p})\times \kappa(\acute{q})}}=\bigg(\frac{m}{m-n+2}\bigg)\sum_{\acute{p}\acute{q} \in \mathrm{E_j}(\mathrm{G})}\sum_{j=1}^9\\&&\frac{1}{\sqrt{\kappa(\acute{p})\times \kappa(\acute{q})}}
\end{eqnarray*}
\begin{eqnarray*}
J(\mathrm{G_3})&=&\bigg(\frac{19-40r+21r^2}{15-28r+14r^2}\bigg)\bigg(\frac{1}{4}|E_1(\mathrm{G_3})|+\frac{1}{2\sqrt{7}}|E_2(\mathrm{G_3})|+
\frac{1}{2\sqrt{10}}|E_3(\mathrm{G_3})|+\\&&\frac{1}{6\sqrt{2}}|E_4(\mathrm{G_3})|+\frac{1}{\sqrt{70}}
|E_5(\mathrm{G_3})|+\frac{1}{3\sqrt{14}}|E_6(\mathrm{G_3})|+\frac{1}{10}|E_7(\mathrm{G_3})|+\\&&
\frac{1}{6\sqrt{5}}|E_8(\mathrm{G_3})|+\frac{1}{18}|E_9(\mathrm{G_3})|\bigg)
\end{eqnarray*}
By doing some calculation, we get
\begin{eqnarray*}
\Longrightarrow J(\mathrm{G_3})&=&\bigg(\frac{1}{315(15-28r+14r^2)}\bigg)7(-157-180\sqrt{2}+12\sqrt{5}+54\sqrt{10})r+\\&&105(5+3\sqrt{2})r^2)
(19-40r+21r^2)(3(280+420\sqrt{2}-70\sqrt{5}+\\&&60\sqrt{7}-231\sqrt{10}+5\sqrt{14}+6\sqrt{70})).
\end{eqnarray*}
\qed
\end{proof}
Now, we compute $ReZG1$, $ReZG2$ and $ReZG3$ indices of triangular hex derived network $\mathrm{G_3}$.
\begin{theorem}
Let $\mathrm{G_3}$ be the rectangular hex derived network, then\\ \\
$\bullet$ $ReZG1(\mathrm{G_3})$ = $\frac{10102843}{32725}-\frac{2036r}{11}+\frac{861n^2}{11}$; \\ \\
$\bullet$ $ReZG2(\mathrm{G_3})$ = $56-12r+7r^2$; \\ \\
$\bullet$ $ReZG3(\mathrm{G_3})$ = $4(50785-50608r+13692r^2)$.
\end{theorem}
\begin{proof}
By using edge partition given in Table 3, the ReZG1$(\mathrm{G_3})$ can be calculated by using $(4)$ as follows.
\begin{eqnarray*}
ReZG1(\mathrm{G})&=&\sum_{\acute{p}\acute{q} \in \mathrm{E}(\mathrm{G})}\bigg(\frac{\kappa(\acute{p})\times \kappa(\acute{q})}{\kappa(\acute{p})+ \kappa(\acute{p})}\bigg)=\sum_{j=1}^9\sum_{\acute{p}\acute{q} \in E_j(\mathrm{G_3})}\bigg(\frac{\kappa(\acute{p})\times \kappa(\acute{q})}{\kappa(\acute{p})+ \kappa(\acute{p})}\bigg)
\end{eqnarray*}

\begin{eqnarray*}
ReZG1(\mathrm{G_3})&=&2|E_1(\mathrm{G_3})|+\frac{28}{11}|E_2(\mathrm{G_3})|+
\frac{20}{7}|E_3(\mathrm{G_3})|+\frac{36}{11}|E_4(\mathrm{G_3})|
+\\&&\frac{70}{17}|E_5(\mathrm{G_3})|+\frac{126}{25}|E_6(\mathrm{G_3})|+5|E_7(\mathrm{G_3})|+
\frac{45}{7}|E_8(\mathrm{G_3})|+\\&&9|E_9(\mathrm{G_3})|
\end{eqnarray*}
By doing some calculation, we get\\
$$\Longrightarrow ReZG1(\mathrm{G_3})=\frac{10102843}{32725}-\frac{2036r}{11}+\frac{861n^2}{11}$$
The ReZG2$(\mathrm{G_3})$ can be calculated by using $(5)$ as follows.
$$ReZG2(\mathrm{G})=\sum_{\acute{p}\acute{q} \in \mathrm{E}(\mathrm{G})}\bigg(\frac{\kappa(\acute{p})+\kappa(\acute{q})}{\kappa(\acute{p})\times \kappa(\acute{p})}\bigg)=\sum_{j=1}^9\sum_{\acute{p}\acute{q} \in E_j(\mathrm{G_3})}\bigg(\frac{\kappa(\acute{p})+\kappa(\acute{q})}{\kappa(\acute{p})\times \kappa(\acute{p})}\bigg)$$
\begin{eqnarray*}
ReZG2(\mathrm{G_3})&=&\frac{1}{2}|E_1(\mathrm{G_3})|+\frac{11}{28}|E_2(\mathrm{G_3})|+
\frac{7}{20}|E_3(\mathrm{G_3})|+
\frac{11}{36}|E_4(\mathrm{G_3})|
+\\&&\frac{17}{70}|E_5(\mathrm{G_3})|+\frac{25}{126}|E_6(\mathrm{G_3})|+\frac{1}{5}|E_7(\mathrm{G_3})|+\frac{7}{45}|E_8(\mathrm{G_3})|
+\\&&\frac{1}{9}|E_9(\mathrm{G_3})|
\end{eqnarray*}
By doing some calculation, we get\\
$$\Longrightarrow ReZG2(\mathrm{G_3})=56-12r+7r^2$$
The ReZG3$(\mathrm{G_3})$ index can be calculated from (6) as follows.
\begin{eqnarray*}
ReZG3(\mathrm{G})&=&\sum_{\acute{p}\acute{q} \in \mathrm{E}(\mathrm{G})}(\kappa(\acute{p})\times\kappa(\acute{q}))(\kappa(\acute{p})+ \kappa(\acute{p})\\&&=\sum_{j=1}^{6}\sum_{\acute{p}\acute{q} \in E_j(\mathrm{G_3})}(\kappa(\acute{p})\times\kappa(\acute{q}))(\kappa(\acute{p})+ \kappa(\acute{p}))
\end{eqnarray*}
\begin{eqnarray*}
ReZG3(\mathrm{G_3})&=&128|E_1(\mathrm{G_3})|+308|E_2(\mathrm{G_3})|+560|E_3(\mathrm{G_3})|+1584|E_4(\mathrm{G_3})|
+\\&&1190|E_5(\mathrm{G_3})|+3150|E_6(\mathrm{G_3})|+2000|E_7(\mathrm{G_3})|+
5040|E_8(\mathrm{G_3})|+\\&&11664|E_9(\mathrm{G_3})|
\end{eqnarray*}
By doing some calculation, we get\\
\begin{eqnarray*}
\Longrightarrow ReZG3(\mathrm{G_3})&=&4(50785-50608r+13692r^2)
\end{eqnarray*}
\end{proof}
\qed
\begin{table}[h]
\centering
\begin{tabular}{|c|c|c|c|c|}
   \hline
  $(\kappa_{x},\kappa_{y})$ where $\acute{p}\acute{q}\in E(\mathrm{G_3})$ & $\text{Number of edges}$&$(\kappa_{u},\kappa_{v})$ where $\acute{p}\acute{q}\in E(\mathrm{G_3})$&$\text{Number of edges}$ \\\hline
   $(22,22)$ & $2$&$(44,44)$ & $6r^2-36r+54$ \\\hline
   $(22,28)$ & $8$&$(44,124)$ & $8$ \\\hline
   $(22,63)$ & $4$&$(44,129)$ & $12$\\\hline
   $(25,33)$ & $4$& $(44,140)$ & $32r-128$\\\hline
   $(25,36)$ & $4$& $(44,156)$ & $12r^2-96r+192$\\\hline
    $(25,54)$ & $4$&$(54,63)$ & $4$ \\\hline
    $(25,63)$ & $4$&$(54,129)$ & $2$ \\\hline
   $(28,36)$ & $8r-20$&$(63,63)$ & $4r-10$ \\\hline
   $(28,63)$ & $8r-12$&$(63,124)$ & $8$\\\hline
   $(33,36)$ & $4$& $(63,129)$ & $4$\\\hline
   $(33,54)$ & $4$& $(63,140)$ & $8r-32$\\\hline
   $(33,129)$ & $4$&$(124,140)$ & $4$ \\\hline
   $(36,36)$ & $8r-22$&$(129,140)$ & $4$\\\hline
     $(36,44)$ & $8r-24$&$(129,156)$ & $2$\\\hline
   $(36,63)$ & $16r-40$& $(140,140)$ & $4r-18$\\\hline
   $(36,124)$ & $16$& $(140,156)$ & $8r-36$\\\hline
   $(36,129)$ & $8$&$(156,156)$ & $3r^2-28r+65$ \\\hline
   $(36,140)$ & $16r-64$& &\\\hline
\end{tabular}
\vspace{.2cm}
\caption{Edge partition of Hex Derived network of type 3 $HDN3(r)$ based on degrees of end vertices of each edge.}
\end{table}
Now, we compute $ABC_{4}$ and $GA_{5}$ indices of triangular hex derived network $\mathrm{G_3}$.
\begin{theorem}
Let $\mathrm{G_3}$ be the triangular hex derived network, then\\ \\
$\bullet$ $ABC_{4}(\mathrm{G_3})$ = $22.459 +8\sqrt{\frac{26}{55}}(-4+r)+4\sqrt{\frac{58}{105}}(-4+r)+\frac{4}{7}\sqrt{\frac{67}{15}}(-4+r)+3\sqrt{\frac{6}{13}} (-4+r)^2+2\sqrt{\frac{26}{33}}(-3+r)+\frac{3}{11}\sqrt{\frac{43}{2}}(-3+r)^2+\sqrt{\frac{14}{65}}(-9+2r)+\frac{1}{35}\sqrt{\frac{139}{2}}
(-9+2r)+\frac{1}{3}\sqrt{\frac{62}{7}}(-5+2r)+\frac{4}{63}\sqrt{31}(-5+2r)+\frac{4}{9}\sqrt{\frac{97}{7}}(-3+2r)+\frac{2}{21}
\sqrt{89}(-3+2r)+\frac{1}{9}\sqrt{\frac{35}{2}}(-11+4r)+\frac{1}{78}\sqrt{\frac{155}{2}}(65-28r+3r^2)$; \\ \\
$\bullet$ $GA_5(\mathrm{G_3})$ = $173.339+\frac{96}{29}\sqrt{5}(-4+r)+\frac{24}{11}\sqrt{35}(-4+r)+\frac{32}{23}\sqrt{385}(-4+r)+\frac{12}{25}\sqrt{429}(-4+r)^2+\frac{12}{5} \sqrt{11}(-3+r)-48r+9r^2+\frac{4}{37}\sqrt{1365}(-9+2r)+\frac{3}{2}\sqrt{7}(-5+2r)+\frac{48}{13}(-3+2r)+\frac{32}{11}\sqrt{7}(-3+2r)$.
\end{theorem}
\begin{proof}
By using edge partition given in Table 4, the $ABC_{4}(\mathrm{G_3})$ can be calculated by using $(7)$ as follows.
$$ABC_{4}(\mathrm{G})=\sum_{\acute{p}\acute{q} \in \mathrm{E}(\mathrm{G})}\sqrt{\frac{S_{\acute{p}}+S_{\acute{q}}-2}{S_{\acute{p}}S_{\acute{q}}}}=\sum_{j=10}^{44}\sum_{\acute{p}\acute{q} \in E_j(\mathrm{G_3})}\sqrt{\frac{S_{\acute{p}}+S_{\acute{q}}-2}{S_{\acute{p}}S_{\acute{q}}}}$$
\begin{eqnarray*}
ABC_4(\mathrm{G_3})&=&\frac{1}{11}\sqrt{\frac{21}{2}}|E_{10}(\mathrm{G_3})|+
\sqrt{\frac{6}{77}}|E_{11}(\mathrm{G_3})|+\frac{1}{3}\sqrt{\frac{83}{154}}|E_{12}(\mathrm{G_3})|+\\&&
\frac{1}{5}\sqrt{\frac{46}{33}}|E_{13}(\mathrm{G_3})|+\frac{1}{30}\sqrt{59}|E_{14}(\mathrm{G_3})|+
\frac{1}{15}\sqrt{\frac{77}{6}}|E_{15}(\mathrm{G_3})|+\\&&
\frac{1}{15}\sqrt{\frac{86}{7}}|E_{16}(\mathrm{G_3})|+\frac{1}{6}\sqrt{\frac{31}{14}}|E_{17}(\mathrm{G_3})|+
\frac{1}{42}\sqrt{89}|E_{18}(\mathrm{G_3})|+\\&&
\frac{1}{6}\sqrt{\frac{67}{33}}|E_{19}(\mathrm{G_3})|+\frac{1}{9}\sqrt{\frac{85}{22}}|E_{20}(\mathrm{G_3})|+
\frac{4}{3}\sqrt{\frac{10}{473}}|E_{21}(\mathrm{G_3})|+\\&&
\frac{1}{18}\sqrt{\frac{35}{2}}|E_{22}(\mathrm{G_3})|+\frac{1}{2}\sqrt{\frac{13}{66}}|E_{23}(\mathrm{G_3})|+
\frac{1}{18}\sqrt{\frac{97}{7}}|E_{24}(\mathrm{G_3})|+\\&&\frac{1}{6}\sqrt{\frac{79}{62}}|E_{25}(\mathrm{G_3})|
+\frac{1}{6}\sqrt{\frac{163}{129}}|E_{26}(\mathrm{G_3})|
+\frac{1}{2}\sqrt{\frac{29}{210}}|E_{27}(\mathrm{G_3})|+\\&&
\frac{1}{22}\sqrt{\frac{43}{2}}|E_{28}(\mathrm{G_3})|
+\frac{1}{2}\sqrt{\frac{83}{682}}|E_{29}(\mathrm{G_3})|
+\frac{1}{2}\sqrt{\frac{57}{473}}|E_{30}(\mathrm{G_3})|
+\\&&\frac{1}{2}\sqrt{\frac{13}{110}}|E_{31}(\mathrm{G_3})|
+\frac{1}{2}\sqrt{\frac{3}{26}}|E_{32}(\mathrm{G_3})
+\frac{1}{9}\sqrt{\frac{115}{42}}|E_{33}(\mathrm{G_3})
+\\&&\frac{1}{9}\sqrt{\frac{181}{86}}|E_{34}(\mathrm{G_3})|
+\frac{1}{63}\sqrt{31}|E_{35}(\mathrm{G_3})|
+\frac{1}{6}\sqrt{\frac{185}{217}}|E_{36}(\mathrm{G_3})|
+\\&&\frac{1}{3}\sqrt{\frac{190}{903}}|E_{37}(\mathrm{G_3})|
+\frac{1}{14}\sqrt{\frac{67}{15}}|E_{38}(\mathrm{G_3})|
+\frac{1}{2}\sqrt{\frac{131}{2170}}|E_{39}(\mathrm{G_3})|
+\\&&\frac{1}{2}\sqrt{\frac{89}{1505}}|E_{40}(\mathrm{G_3})
+\frac{1}{70}\sqrt{\frac{283}{559}}|E_{41}(\mathrm{G_3})
+\frac{1}{70}\sqrt{\frac{139}{2}}|E_{42}(\mathrm{G_3})|
+\\&&\frac{1}{2}\sqrt{\frac{7}{130}}|E_{43}(\mathrm{G_3})|
+\frac{1}{78}\sqrt{\frac{155}{2}}|E_{44}(\mathrm{G_3})|
\end{eqnarray*}
By doing some calculation, we get\\
\begin{eqnarray*}
\Longrightarrow ABC_4(\mathrm{G_3})&=&22.459 +8\sqrt{\frac{26}{55}}(-4+r)+4\sqrt{\frac{58}{105}}(-4+r)+\frac{4}{7}\sqrt{\frac{67}{15}}\\&&(-4+r)+3\sqrt{\frac{6}{13}} (-4+r)^2+2\sqrt{\frac{26}{33}}(-3+r)+\frac{3}{11}\sqrt{\frac{43}{2}}\\&&(-3+r)^2+\sqrt{\frac{14}{65}}(-9+2r)+\frac{1}{35}\sqrt{\frac{139}{2}}
(-9+2r)+\frac{1}{3}\sqrt{\frac{62}{7}}\\&&(-5+2r)+\frac{4}{63}\sqrt{31}(-5+2r)+\frac{4}{9}\sqrt{\frac{97}{7}}(-3+2r)+\frac{2}{21}
\sqrt{89}\\&&(-3+2r)+\frac{1}{9}\sqrt{\frac{35}{2}}(-11+4r)+\frac{1}{78}\sqrt{\frac{155}{2}}(65-28r+3r^2)
\end{eqnarray*}
The $GA_{5}(\mathrm{G_3})$ index can be calculated from (8) as follows.
\begin{eqnarray*}
GA_{5}(\mathrm{G})=\sum_{\acute{p}\acute{q} \in \mathrm{E}(\mathrm{G})}\frac{2\sqrt{S_{\acute{p}}S_{\acute{q}}}}{(S_{\acute{p}}+S_{\acute{q}})}=\sum_{j=10}^{44}\sum_{\acute{p}\acute{q} \in E_j(\mathrm{G_3})}\frac{2\sqrt{S_{\acute{p}}S_{\acute{q}}}}{(S_{\acute{p}}+S_{\acute{q}})}
\end{eqnarray*}
\begin{eqnarray*}
GA_5(\mathrm{G_3})&=&1|E_{10}(\mathrm{G_3})|+\frac{2}{25}\sqrt{154}|E_{11}(\mathrm{G_3})|+
\frac{6}{85}\sqrt{154}|E_{12}(\mathrm{G_3})|+\\&&
\frac{5}{29}\sqrt{33}|E_{13}(\mathrm{G_3})|+\frac{60}{61}|E_{14}(\mathrm{G_3})|+
\frac{30}{79}\sqrt{6}|E_{15}(\mathrm{G_3})|+\\&&
\frac{15}{44}\sqrt{7}|E_{16}(\mathrm{G_3})|+\frac{3}{8}\sqrt{7}|E_{17}(\mathrm{G_3})|+
\frac{12}{13}|E_{18}(\mathrm{G_3})|
+\frac{4}{23}\sqrt{33}|E_{19}(\mathrm{G_3})|+\\&&\frac{6}{29}\sqrt{22}|E_{20}(\mathrm{G_3})|+
\frac{1}{27}\sqrt{473}|E_{21}(\mathrm{G_3})|
+1|E_{22}(\mathrm{G_3})|+\\&&\frac{3}{10}\sqrt{11}|E_{23}(\mathrm{G_3})|+
\frac{4}{11}\sqrt{7}|E_{24}(\mathrm{G_3})|+
\frac{3}{20}\sqrt{31}|E_{25}(\mathrm{G_3})|+\\&&\frac{4}{55}\sqrt{129}|E_{26}(\mathrm{G_3})|+
\frac{3}{22}\sqrt{35}|E_{27}(\mathrm{G_3})|
+|E_{28}(\mathrm{G_3})|
+\\&&\frac{1}{21}\sqrt{341}|E_{29}(\mathrm{G_3})|
+\frac{4}{173}\sqrt{1419}|E_{30}(\mathrm{G_3})|
+\frac{1}{23}\sqrt{385}|E_{31}(\mathrm{G_3})|
+\\&&\frac{1}{25}\sqrt{429}|E_{32}(\mathrm{G_3})|
+\frac{2}{13}\sqrt{42}|E_{33}(\mathrm{G_3})|
+\frac{6}{61}\sqrt{86}|E_{34}(\mathrm{G_3})|
+1|E_{35}(\mathrm{G_3})|
+\\&&\frac{12}{187}\sqrt{217}|E_{36}(\mathrm{G_3})|
+\frac{1}{32}\sqrt{903}|E_{37}(\mathrm{G_3})|
+\frac{12}{29}\sqrt{5}|E_{38}(\mathrm{G_3})|
+\\&&\frac{1}{33}\sqrt{1085}|E_{39}(\mathrm{G_3})|
+\frac{4}{269}\sqrt{4515}|E_{40}(\mathrm{G_3})|
+\frac{4}{95}\sqrt{559}|E_{41}(\mathrm{G_3})|
+\\&&1|E_{42}(\mathrm{G_3})|
+\frac{1}{37}\sqrt{1365}|E_{43}(\mathrm{G_3})|
+1|E_{44}(\mathrm{G_3})|
\end{eqnarray*}
By doing some calculation, we get\\
\begin{eqnarray*}
\Longrightarrow GA_5(\mathrm{G_3})&=& 173.339+\frac{96}{29}\sqrt{5}(-4+r)+\frac{24}{11}\sqrt{35}(-4+r)+\frac{32}{23}\sqrt{385}(-4+r)+\\&&\frac{12}{25}\sqrt{429}(-4+r)^2+\frac{12}{5} \sqrt{11}(-3+r)-48r+9r^2+\\&&\frac{4}{37}\sqrt{1365}(-9+2r)+\frac{3}{2}\sqrt{7}(-5+2r)+\frac{48}{13}(-3+2r)+\\&&\frac{32}{11}\sqrt{7}(-3+2r)
\end{eqnarray*}
\end{proof}
\section{Conclusion}
In this paper, we have studied newly formed third type of hex derived networks, $HDN3$, $THDN3$ and $RHDN3$. The exact results have been computed of Randi$\acute{c}$, Zagreb, Harmonic, Augmented Zagreb, atom-bond connectivity and Geometric-Arithmetic indices for the very first time of third type of hex-derived networks also find the numerical computation for all the networks. As these important results are help in many chemical point of view as well as for pharmaceutical sciences. We are looking forward in future to derived and compute new networks and topological indices.

\end{document}